\documentclass[english,11pt]{amsart}

\usepackage[utf8]{inputenc}
\usepackage{amssymb,babel,mathtools,xcolor}
\usepackage{hyperref}
\hypersetup{colorlinks,allcolors=[rgb]{0,0,0.5},
    pdftitle={On the level of a Calabi-Yau hypersurface},
    pdfauthor={Stiofáin Fordham}}

\newtheorem{thm}{Theorem}

\newtheorem{prop}[thm]{Proposition}
\newtheorem{cor}[thm]{Corollary}
\theoremstyle{definition}
\newtheorem{dfn}[thm]{Definition}
\newtheorem{rem}[thm]{Remark}

\newcommand{\F}{\mathbf{F}}

\newcommand{\oo}{\mathcal{O}}
\newcommand{\Proj}{\mathrm{Proj}}

\begin{document}
\title{On the level of a Calabi--Yau hypersurface}
\date{}
\author{Stiofáin Fordham}
\address{School of Mathematics and Statistics, University College Dublin,
    Dublin D04 N2E5, Ireland.}
\thanks{The author thanks Adrian Langer for a number of useful comments,
    and also A.F.\ Boix, O.\ Gregory, M.\ Katzman. The author is supported
    by SFI grant 13/IA/1914.}
\email{stiofain.fordham@ucdconnect.ie}
\begin{abstract}
    Boix--De Stefani--Vanzo defined the notion of level for a smooth projective
    hypersurface over a finite field in terms of the stabilisation of a chain
    of ideals previously considered by \`Alvarez-Montaner--Blickle--Lyubeznik,
    and showed that in the case of an elliptic curve the level is 1 if and only
    if it is ordinary and 2 otherwise. Here we extend their theorem to the case
    of Calabi--Yau hypersurfaces by relating their level to the $F$-jumping
    exponents of Blickle--Mustaţă--Smith and the Hartshorne--Speiser--Lyubeznik
    numbers of Mustaţă--Zhang.
\end{abstract}
\maketitle

%%%%%%%%%%%%%%%%%%%%%%%%%%%%%%%%%%%%%%%%%%%%%%%%%%%%%%%%%%%%%%%%%%%%%%%%%%%%%%
%%%%%%%%%%%%%%%%%%%%%%%%%%%%%%%%%%%%%%%%%%%%%%%%%%%%%%%%%%%%%%%%%%%%%%%%%%%%%%
%%%%%%%%%%%%%%%%%%%%%%%%%%%%%%%%%%%%%%%%%%%%%%%%%%%%%%%%%%%%%%%%%%%%%%%%%%%%%%

\section{Introduction}

Let $R$ be a finitely-generated $k$-algebra where $k$ is a perfect field of
characteristic $p>2$ and let $F^e\colon R \rightarrow
R$ denote the Frobenius morphism $a\mapsto a^{p^e}$. Let $F_\ast^e R$ denote
$R$ with the $R$-module structure $r_1\cdot r_2\coloneqq r_1^{p^e}r_2$, then
$F^e\colon R \rightarrow F_\ast^e R$ is an $R$-module homomorphism. A ring $R$
is said to {\it $F$-finite} if $F_\ast R$ is finitely-generated as an
$R$-module.

Let $D_R$ denote the ring of $k$-linear differential operators on $R$ (see
\cite[IV, \S16]{ega4} or \cite[Chp.\ 1]{yeku} for definition and further
details). Each element $\delta \in D_R$ of order $\le p^e-1$ is $F_\ast^e
R$-linear, so in particular
\begin{equation}\label{eqn:Dlinear}
D_R\subseteq    \bigcup_{e\ge 1} D_R^{(e)},
\end{equation}
where $D_R^{(e)}=\mathrm{End}_{R^{p^e}}(R)$, and since $R$ is a finitely-generated
$k$-algebra with $k$ perfect, then $R$ is $F$-finite and thus
(\ref{eqn:Dlinear}) is actually an equality \cite[Chp.\ 1, \S4]{yeku}.

For an ideal $\mathfrak b\subset R$ and integer $e>0$, write $\mathfrak
b^{[1/p^e]}$ for the smallest ideal $J$ such that $\mathfrak b \subseteq
J^{[p^e]}$ where $J^{[a]}\coloneqq (x^a\colon x\in J)$\footnote{The equivalent
notation for $\mathfrak{b}^{[1/p^e]}$ used in \cite{ambl} and \cite{bdsv}
is $I_e(\mathfrak{b})$---the notation used here is the one of \cite{bms08}
and \cite{bms09}.}. For an element $f\in R$, we have by \cite[Lem.\ 3.4]{ambl}
a descending chain of ideals
\begin{equation}\label{eqn:chain}
    R=(f^{p^0-1})^{[1/p^0]}\supseteq(f^{p-1})^{[1/p]} \supseteq
    (f^{p^2-1})^{[1/p^2]}\supseteq (f^{p^3-1})^{[1/p^3]}\supseteq \ldots,
\end{equation}
that stabilises \cite[Thm.\ 3.7]{ambl}, and it stabilises at $e$ (i.e.\
$(f^{p^e-1})^{[1/p^e]}=(f^{p^{e+1}-1})^{[1/p^{e+1}]}=\ldots$) if and only if
there exists $\delta\in D_R^{(e+1)}$ such that $\delta(f^{p^e-1})=f^{p^e-p}$
\cite[Prop.\ 3.5]{ambl}.

\begin{rem}
    \`Alvarez-Montaner--Blickle--Lyubeznik \cite{ambl} considered the above
    chain of ideals in connection with the question of determining the minimum
    integer $i$ such that $1/f^i$ generates $R_f$ as a $D_R$-module which is
    known to be finite (where the $D_R$-structure on $R_f$ is `via the
    quotient-rule'). In the case of a field of characteristic 0 and $f$ a
    non-zero polynomial, then the integer $i$ is related in a non-trivial way
    to the Bernstein--Sato polynomial for $f$. In the case of a
    field of positive characteristic and $f\in R$ a non-zero polynomial, then
    one always has $i=1$ \cite[Thm.\ 1.1]{ambl}.
\end{rem}

\begin{dfn}
    Let $R$ be a finitely-generated $k$-algebra where $k$ is a perfect field of
    characteristic $p>2$. The {\it level} of $f\in R$ is defined to be $e+1$,
    where $e$ is the integer where the chain (\ref{eqn:chain}) stabilises. 
\end{dfn}

Given $f\in R$ a homogeneous polynomial, one can consider the projective
hypersurface $X$ defined by the vanishing of $f$. The level of $f$ was shown to
be connected with the geometry of $X$ in the case of an elliptic curve, in
the following

\begin{thm}[Boix--De Stefani--Vanzo, {\cite[Thm.\ 1.1]{bdsv}}]\label{bdsv}
    Let $R=k[x_0,x_1,x_2]$ where $k$ is a perfect field of characteristic $p>2$,
    and let $X=\mathrm{Proj}(R/fR)$ define an elliptic curve. Then the level of
    $f$ is 1 if and only if $X$ is ordinary, and 2 otherwise.
\end{thm}

\begin{rem}
    The above result has been partially generalised to the case of
    hyperelliptic curves in \cite{bbfy}.
\end{rem}

The objective here is to extend the above theorem to the case of Calabi--Yau
hypersurfaces, and in the process to describe the relation between level and
the $F$-jumping exponents and the Hartshorne--Speiser--Lyubeznik numbers.

%%%%%%%%%%%%%%%%%%%%%%%%%%%%%%%%%%%%%%%%%%%%%%%%%%%%%%%%%%%%%%%%%%%%%%%%%%%%%%
%%%%%%%%%%%%%%%%%%%%%%%%%%%%%%%%%%%%%%%%%%%%%%%%%%%%%%%%%%%%%%%%%%%%%%%%%%%%%%
%%%%%%%%%%%%%%%%%%%%%%%%%%%%%%%%%%%%%%%%%%%%%%%%%%%%%%%%%%%%%%%%%%%%%%%%%%%%%%

\section{\texorpdfstring{$F$}{F}-jumping exponents and the level of a
Calabi--Yau variety}
Here we outline the relationship between the level of a polynomial $f$ and the 
$F$-jumping exponents of the ideal of $f$.

Let $R$ be a regular ring of characteristic $p>0$. For an ideal $\mathfrak a$
in $R$ and positive real number $\lambda$, and $e\ge1$ then by \cite[Lem.\
2.8]{bms08} we have
\[
    (\mathfrak a^{\lceil \lambda p^e \rceil})^{[1/p^e]} \subseteq (\mathfrak
    a^{\lceil \lambda p^{e+1} \rceil} )^{[1/p^{e+1}]},
\]
where for a real number $r$, we set $\lceil r \rceil$ to be the smallest
integer $\ge r$. The ring $R$ is Noetherian so the chain stabilises to a limit
ideal for large $e$ and we set $\tau(\mathfrak{a}^\lambda)$ to be this limit
called the {\it generalised test ideal of $\mathfrak a$ with exponent
$\lambda$}. It is known that for every $\lambda$ there is $\varepsilon>0$ such
that $\tau(\mathfrak a^\lambda)=\tau(\mathfrak a^{\lambda'})$ for $\lambda'\in
[\lambda, \lambda+\varepsilon)$ \cite[Cor.\ 2.16]{bms08}. A positive $\lambda$
is called an {\it $F$-jumping exponent} of $\mathfrak a$ if $\tau(\mathfrak
a^\lambda)\ne \tau(\mathfrak a^{\lambda'})$ for all $\lambda'<\lambda$ (if $\mu
< \lambda$ then one has $\tau(\mathfrak a^\mu)\supseteq \tau(\mathfrak
a^\lambda)$). Call an $F$-jumping exponent {\it simple} if it lies in the
interval $(0,1]$. If $\mathfrak a$ is a principal ideal then it follows from
\cite[Prop.\ 2.25]{bms08} that
$\tau(\mathfrak{a}^\lambda)=\tau(\mathfrak{a}^{\lambda+1})$ so in our situation
it is enough to look at the simple $F$-jumping exponents.

\begin{rem}
    These should be regarded as the characteristic-$p$ analogues of the related
    notions from birational geometry, e.g.\ the log canonical threshold, the
    jumping coefficients of Ein--Lazarsfeld--Smith--Varolin \cite{elsv} etc.
\end{rem}

\begin{prop}\label{fje}
  Let $R$ be an $F$-finite ring of characteristic $p>2$ and let $I=(f)$ be a
  principal ideal. Let $e$ be the largest simple $F$-jumping exponent of $I$.
  Then the level of $f$ is $\lceil 1-\log_p (1-e)\rceil$.
\end{prop}

\begin{proof}
  Set $\lambda_k=1-\frac{1}{p^k}$, then
  $\tau(I^{\lambda_k})=(f^{p^k-1})^{[1/p^k]}$ \cite[Lem.\ 2.1]{bms09}. Since
  the level is well-defined (i.e.\
  $(f^{p^k-1})^{[1/p^k]}=(f^{p^{k+1}-1})^{[1/p^{k+1}]}$ for $k\gg 0$), thus
  $\tau(I^{\lambda_k})=\tau(I^{\lambda_{k+1}})$ for $k\gg 0$ hence the number
  $e$ is well-defined. Let the level of $f$ be $m$, so
  $(f^{p^{m-1}-1})^{[1/p^{m-1}]}=(f^{{p^m}-1})^{[1/p^m]}$ and thus
  $\tau(\mathfrak a^{\lambda_{m-1}}) = \tau(\mathfrak a^{\lambda_m})$. Hence
  there is an $F$-jumping exponent in the interval
  $(\lambda_{m-2},\lambda_{m-1}]$ and none in $(\lambda_{m-1},1]$. It is
  straightforward to see that the function $f(x)\coloneqq \lceil
  1-\log_p(1-x)\rceil$ is constant on intervals $(\lambda_{m-1},\lambda_m]$ and
  equal to $m+1$. The result follows.
\end{proof}

\begin{rem}\label{rem}
  In particular, if the level of $f$ is 1 then there is no $F$-jumping exponent
  for $I=(f)$ in $(0,1)$ and if the level of $f$ is $>2$ then there is an
  $F$-jumping exponent for $I$ in $(1-1/p,1)$.
\end{rem}

\subsection{Calabi--Yau hypersurfaces} 
A {\it Calabi--Yau variety} is a smooth projective variety $X$ of
dimension $n$, over a field of characteristic $p>0$ with trivial canonical
bundle such that $\dim \mathrm{H}^i(X,\oo_X)=0$ for $i=1,\dots,n-1$.
Following Artin--Mazur \cite{am}, one considers the functor
$\mathrm{AM}_X\colon \mathrm{Art}_k\rightarrow \mathrm{Ab}$ defined by
\[
    \mathrm{AM}_X(S)\coloneqq \ker\big(F_\ast\colon
    \mathrm{H}_{\text{\'et}}^n(X\times S,\mathbb{G}_m)\rightarrow
    \mathrm{H}_{\text{\'et}}^n(X,\mathbb{G}_m)\big),
\]
on the category $\mathrm{Art}_k$ of local Artinian $k$-algebras $S$ with
residue field $k$. The functor is pro-representable by a smooth formal group of
dimension 1 \cite[II, \S2]{am}, which is characterised by a number called the
{\it height} which can be any positive integer or infinity. We say that a
Calabi--Yau variety is {\it ordinary} if the aforementioned formal group has
height 1. In the case that $X$ is an elliptic curve then this agrees with the
other definition.

\begin{prop}\label{gen}
    Let $X=\mathrm{Proj}(R/fR)$ be a Calabi--Yau hypersurface over a perfect
    field of characteristic $p>0$. Then $f$ has level 1 if and only if $X$ is
    ordinary.
\end{prop}

\begin{proof}
     If $f$ has level one then from remark \ref{rem} there is no $F$-jumping
     number in $(0,1)$ hence by \cite[Thm.\ 1.1]{bs} $X$ is ordinary and the reverse
     implication holds as well.
\end{proof}

%%%%%%%%%%%%%%%%%%%%%%%%%%%%%%%%%%%%%%%%%%%%%%%%%%%%%%%%%%%%%%%%%%%%%%%%%%%%%%
%%%%%%%%%%%%%%%%%%%%%%%%%%%%%%%%%%%%%%%%%%%%%%%%%%%%%%%%%%%%%%%%%%%%%%%%%%%%%%
%%%%%%%%%%%%%%%%%%%%%%%%%%%%%%%%%%%%%%%%%%%%%%%%%%%%%%%%%%%%%%%%%%%%%%%%%%%%%%

\section{The Hartshorne--Speiser--Lyubeznik number}

\begin{dfn}[\cite{MusZha}]
  Let $R$ be a Noetherian local ring of characteristic $p>0$, let $M$ be an
  Artinian $R$-module and let $\varphi\colon M \rightarrow M$ be an additive
  map satisfying $\varphi(am)=a^p\varphi(m)$ for $a\in R$ and $m\in M$. For
  integers $i\ge1$, define a chain of ascending submodules
  \[
    N_i=\{z\in M\colon \varphi^i(z)=0\},
  \]
  then a theorem of Lyubeznik \cite[Prop.\ 4.4]{Ly97} implies that the chain
  stabilises: $N_\ell=N_{\ell+1}$ for $\ell\gg 1$. The {\it
  Hartshorne--Speiser--Lyubeznik number} of the pair $(M,\varphi)$ is the
  smallest positive integer $\ell$ such that $N_\ell=N_{\ell+j}$ for all $j\ge
  1$.
\end{dfn}

Now let $R=k[x_0,\dots,x_n]$ where $k$ is a perfect field of characteristic
$p>0$ and let $f\in R$ be non-zero and $\mathfrak m=(x_0,\dots,x_n)$.

\begin{prop}[{\cite[Prop.\ 5.7 and Cor.\ 5.8]{MusZha}}]
  The Hartshorne--Speiser--Lyubeznik number of
  $(\mathrm{H}_\mathfrak{m}^{n+1}(R/fR),\Theta)$ is the smallest positive
  integer $\ell$ such that
  \[
  \tau ((f)^{1-\frac{1}{p^\ell}})=\tau((f)^{1-\frac{1}{p^{\ell+i}}}),
  \]
  for all $i\ge1$, where $\Theta$ denotes the map induced on the top local
  cohomology of $R/fR$ at the origin by the Frobenius morphism.
\end{prop}

\begin{prop}\label{pr:main}
    Let $X=\Proj(R/fR)$ be a Calabi--Yau hypersurface and $p>n^2-n-1$. Then the
    Hartshorne--Speiser--Lyubeznik number of $X$ is 1.
\end{prop}

\begin{proof}
The Grothendieck--Serre correspondence \cite[Thm.\ 20.4.4]{BS13} yields
\[
  \bigoplus_{i\in \mathbf{Z}}
  \mathrm{H}^n(X,\oo_X(i))\cong
  \mathrm{H}^{n+1}_\mathfrak m(R/fR),
\]
which is a graded isomorphism and Frobenius compatible. Assume that $p>n^2-n-1$
then \cite[Thm.\ 3.5]{bs} implies that Frobenius acts injectively on the
negative graded part of the top local cohomology. Hence if one is interested in
the Hartshorne--Speiser--Lyubeznik number of $X$ then one only need consider
the powers of Frobenius acting on $\mathrm H^n(X,\oo_X)$.

If $X$ is ordinary then since the $F$-pure threshold is 1 by \cite[Thm.\
1.1]{bs}, the result follows. Otherwise suppose $X$ has height $h>1$. Then
\cite[Lem.\ 4.4]{vdGK00} implies that the height of $X$ is the smallest integer
$i$ such that the Frobenius map $F\colon \mathrm{H}^n(X,W_i\mathcal
O_X)\rightarrow \mathrm{H}^n(X, W_i\oo_X)$ is non-zero (here $W_i\oo_X$ is the
sheaf of truncated Witt vectors of length $i$ of $\oo_X$)\footnote{The result
proved is for the case of K3 surfaces, but as noted in \cite[Pg.\ 2]{CY},
the same proof goes through in the higher-dimensional case.}.  But when this
map is zero then the maps $F\colon \mathrm{H}^n(X,W_j\mathcal O_X)\rightarrow
\mathrm{H}^n(X, W_j\oo_X)$ are zero for all $1\le j < i$. In particular the
Frobenius action on $\mathrm{H}^n(X,\oo_X)\rightarrow \mathrm{H}^n(X,\oo_X)$ is
zero so the result follows in this case also.
\end{proof}

\begin{cor}
    Let $p>n^2-n-1$ then the simple $F$-jumping exponents of a non-ordinary
    Calabi--Yau hypersurface $X=\Proj(R/fR)$ lie in the interval
    $[1-h/p,1-1/p]$, where $h$ is the order of vanishing of the Hasse invariant
    on the versal deformation space of $X$.
\end{cor}

\begin{proof}
    This is a combination of proposition \ref{pr:main} and \cite[Thm.\
    1.1]{bs}.
\end{proof}

\begin{cor}
    Let $p>n^2-n-1$ then the level of a Calabi--Yau hypersurface
    $X=\Proj(R/fR)$ is 1 if and only if it is ordinary, and 2 otherwise.
\end{cor}

\begin{proof}
    This is a combination of proposition \ref{pr:main} and proposition
    \ref{gen}.
\end{proof}

\begin{rem}
    The connection between Hartshorne--Speiser--Lyubeznik numbers and the level
    described above can be exploited to give another proof of the main results
    of \cite{bbfy}.
\end{rem}

\section{Example of Fermat hypersurfaces}
In this section we describe how to construct the differential operators
guaranteed to exist by proposition \ref{pr:main} for the example of Fermat
hypersurfaces.

In \cite{bdsv} there is described an algorithm to construct an operator of
level~1 or level~2 for an elliptic curve. That algorithm works also in our case
but invokes a script of Katzman--Schwede \cite{KS} that requires two
computations of Gr\"obner bases (it is not exactly clear what the complexity of
the algorithm in \cite{bdsv} is).  The method to be described for the case
of Fermat hypersurfaces can be shown to work also for a general elliptic curve
using degree by degree arguments and has complexity $O(p^2)$. A computer
implementation seems to indicate that it works also for K3 surfaces defined by
smooth quartics in $\mathbf{P}^3$.

Fix $R=k[x_0,x_1,\dots,x_n]$ with $k$ a perfect field of characteristic $p>n\ge
2$. The ring $\mathrm{End}_{R^{p^e}}(R)$ is the ring extension of $R$ generated
by the operators $D_{t,i}$ for $i=0,\dots,n$ and $t=1,\dots,p^e-1$ where (see
\cite[IV, \S16]{ega4})
\[
    D_{t,i}(x_j^s)=\begin{cases}\binom{t}{s}x_i^{s-t} &\text{if $s\ge t$ and
        $i=j$,}\\
        0 &\text{otherwise.}\end{cases}
\]

Now let $f_n(x_0,\dots,x_n):=x_0^{n+1}+x_1^{n+1}+\dots+x_n^{n+1}$, and denote
by $X_{n,p}$ the Fermat hypersurface of Calabi-Yau type
\[ 
X_{n,p}=\mathrm{Proj}(R/f_n R)\subset \mathbf{P}_{\F_p}^n.
\]

Suppose that $p\equiv 1 \pmod {n+1}$ then the monomial $x_0^{p-1}x_1^{p-1}\dots
x_n^{p-1}$ appears in $f_n^{p-1}$ with non-zero coefficient (by an application
of Kummer's theorem for binomial coefficients).  Hence an operator of level 1 for
$f_n$ when $p\equiv 1 \pmod {n+1}$  is (up to a non-zero constant)
\[
    \Psi_1=\prod_{i=0}^n D_{p-1,i}.
\]
Now we will construct an operator of level~2 for all $f_n$. Set
\[
\alpha\coloneqq (n+1)(p^2-1)-n(n+1)p.
\]
For a given $j\in \{0,\dots,n\}$ then the monomial
\[
  m_j\coloneqq  x_0^{(n+1)p}x_1^{(n+1)p}\cdots x_{j-1}^{(n+1)p} x_j^{\alpha}
  x_{j+1}^{(n+1)p}\cdots x_{n}^{(n+1)p},
\]
appears in $f_n^{p^2-1}$ with non-zero coefficient (again by an application of
Kummer's theorem). Set $\beta\coloneqq p^2-1-(n+1)p$ and define an operator
\[
  \delta_j\coloneqq x_0^\beta x_1^\beta\cdots
x_{j-1}^\beta x_j^{2p^2-1-\alpha} x_{j+1}^\beta\cdots
x_n^\beta \circ \prod_{i=0}^n D_{p^2-1,i},
\]
where the $\circ$ indicates precomposition. Then $\delta_j(m_j)=x_j^{p^2}$ (up
to a non-zero constant). Suppose that the monomial
\begin{equation}\label{mono}
x_0^{p^2-p}x_1^{p^2-p}\cdots x_n^{p^2-p},
\end{equation}
does not appear in $f_n^{p^2-p}$. Then by the pigeonhole principle and
$f_n^{p^2-p}=(f_n^p)^{p-1}$, for every monomial appearing in $f_n^{p^2-p}$
there is a variable $x_k$ in that monomial whose power is $\ge p^2$. Hence we
can use the operators $\{\delta_j\}_j$ to construct an operator of level 2 for
$f_n$.  Otherwise, suppose that the monomial (\ref{mono}) does appear in
$f_n^{p^2-p}$ with a non-zero coefficient, then the monomial
$x_0^{p-1}x_1^{p-1}\cdots x_{m-1}^{p-1}$ appears in $f_n^{p-1}$ with non-zero
coefficient and $\Psi_1$ is an operator of level~1, and also of level~2:
\[
  \Psi_1(f_n^{p^2-1})=\Psi_1(f_n^{p^2-p}\cdot
  f_n^{p-1})=f_n^{p^2-p}\Psi_1(f_n^{p-1})=f_n^{p^2-p}.
\]

Explicitly, if we take the case $n=2$ and $p=5$ the above procedure obtains the
following operator of level 2 for $f_2$:
\begin{gather*}
 (x_0^{35}\delta_0+x_1^{35}\delta_1+x_2^{35}\delta_2) +(x_0^5x_1^{30} \delta_0
 +x_1^5x_2^{30} \delta_1+x_2^5 x_0^{30}\delta_2) \\ - ([x_0^{20}x_1^{15} +
 x_0^{20}x_2^{15} ]\delta_0 + [x_0^{15}x_1^{20} + x_1^{20}x_2^{15}]\delta_1 +
 [x_0^{15} x_2^{20} + x_1^{15}x_2^{20}] \delta_2)\\ + 2(x_0^5x_1^{15}x_2^{15}
 \delta_0 + x_0^{15}x_1^5 x_2^{15}\delta_1 + x_0^{15} x_1^{15} x_2^5 \delta_2).
\end{gather*}

\end{document}